\documentclass[10pt]{amsart}

\usepackage{cite,amsthm,amsfonts,amsmath,amssymb,mathrsfs,enumitem,eucal}
\usepackage{accents}
\usepackage[usenames]{color}
\usepackage[colorlinks=true, linkcolor=webgreen, citecolor=webgreen, urlcolor=webbrown]{hyperref}
\usepackage[initials]{amsrefs}
\usepackage{diagrams}
\definecolor{webgreen}{rgb}{0,.4,0}
\definecolor{webbrown}{rgb}{.4,0,0}

\newtheorem{Thm}{Theorem}[section]
\newtheorem{Def}[Thm]{Definition}
\newtheorem{Lm}[Thm]{Lemma}
\newtheorem{Prop}[Thm]{Proposition}
\newtheorem{Cor}[Thm]{Corollary}

\theoremstyle{definition}
\newtheorem*{ack}{Acknowledgements}

\theoremstyle{remark}

\newtheorem{Rem}[Thm]{Remark}

\numberwithin{equation}{section}

\def\<{\langle}
\def\>{\rangle}

\DeclareMathOperator{\GL}{GL}
\DeclareMathOperator{\wtaff}{wt}
\DeclareMathOperator{\wt}{\ring\wtaff}
\DeclareMathOperator{\wtbar}{\overline\wtaff}
\DeclareMathOperator{\ind}{ind}
\DeclareMathOperator{\res}{res}
\DeclareMathOperator{\Ann}{Ann}

\DeclareMathOperator{\head}{head}
\DeclareMathOperator{\rad}{rad}
\DeclareMathOperator{\soc}{soc}
\DeclareMathOperator{\Hom}{Hom}
\DeclareMathOperator{\gr}{gr}

\def\Re{\mathbb R}
\def\Rat{\mathbb Q}
\def\F{\mathbb F}
\def\C{\mathbb C}
\def\Z{\mathbb Z}

\def\a{\alpha}

\def\d{\delta}
\def\th{\theta}

\def\l{\lambda}
\def\L{\Lambda}

\def\R{\mathcal R} % group ring of the finite weight lattice
\def\O{\mathcal O} % minuscule weights
\def\A{\mathbf A} % enveloping algebra of the imaginary Borel
\def\m{\mathbf m} % maximal ideal
\def\Wloc{W_{\rm loc}} % local Weyl module
\def\Gaff{\mathfrak g} % affine Lie algebra
\def\G{\ring{\mathfrak g}} % finite Lie algebra
\def\Haff{\mathfrak h} % affine Cartan
\def\H{\ring {\mathfrak h}} % finite Cartan
\def\Baff{\mathfrak b} % affine Borel
\def\B{\ring{\mathfrak b}} % finite Borel
\def\Naff{\mathfrak n} % affine positive nilpotent
\def\N{\ring{\mathfrak n}} % finite positive nilpotent
\def\K{\mathfrak k} % (hyper)special maximal parabolic
\def\CG{\mathfrak{Cg}} % current algebra
\def\M{\mathfrak F} % category of modules
\def\Aring{\ring A} % finite Cartan matrix
\def\W{\ring W} % finite Weyl group
\def\V{\ring V} % irreducible representations of the finite Lie algebra
\def\D{\ring D} % finite Demazure modules
\def\w{\ring w} % element in finite Weyl group
\def\Rring{\ring R} % finite root system
\def\Rbar{\bar R} % affine root system mod delta
\def\P{\ring P} % finite weight lattice
\def\Q{\ring Q} % finite root lattice
\def\Qbar{\overline{Q}} % projection of the affine root lattice

\begin{document}
\title[]
{BGG reciprocity for current algebras}
\author[]{Vyjayanthi Chari}
\address{Department of Mathematics, University of California, Riverside, CA 92521}
\email{chari@math.ucr.edu}
\author[]{Bogdan Ion}
\address{Department of Mathematics, University of Pittsburgh, Pittsburgh, PA 15260}
\address{Algebra and Number Theory Research Center, Faculty of Mathematics and Computer Science, University of Bucharest,  14 Academiei St., Bucharest, Romania}
\email{bion@pitt.edu}
%\thanks{V.C. was partially supported by DMS-0901253. B.I. was partially supported by  CNCS-UEFISCDI project number 88/05.10.2011.}
\date{May 6, 2014}%
\begin{abstract}
In \cite{BCM} it was conjectured that a BGG-type reciprocity holds for the category of graded representations with finite-dimensional graded components for the current algebra associated to a simple Lie algebra. We associate a current algebra to any indecomposable affine Lie algebra and show that, in this generality, the BGG reciprocity is true for the corresponding category of representations.
\end{abstract}
\maketitle
%%%%%%%%%%%%%%%%%%%%%%%%%%%%%%%%%%%%%%%%%%%%%%%%%%%%%%%%%%%%%%%%%%%%%%%%%%%%%
\section*{Introduction} 

Macdonald polynomials are families of  polynomials that depend on several parameters and are associated to affine roots systems or, equivalently, to indecomposable affine Lie algebras $\Gaff$. They encode a wealth of information at the intersection of combinatorics, algebraic geometry, and representation theory. In particular, it is known that, for $\Gaff$ twisted or simply laced untwisted and after a certain specialization of the parameters, nonsymmetric Macdonald polynomials are the Demazure characters for the level one highest weight $\Gaff$-representations \cites{IonNon, SanCon}. The corresponding statement for the remaining affine Lie algebras is generally false.

In this paper, on one hand, we shall be concerned, for any indecomposable affine Lie algebra $\Gaff$,  with a uniform representation theoretical interpretation of the symmetric Macdonald polynomials after the specialization of parameters alluded to above, and, on the other hand, with establishing a link between combinatorial and algebraic properties of Macdonald polynomials and representation theory. The interpretation is in the context of the representation theory of the current algebra $\CG$ which is defined, up to the scaling operator of $\Gaff$, as the special (hyperspecial if $\Gaff$ is of type $A_{2n}^{(2)}$) maximal parabolic subalgebra $\K$ of $\Gaff$. The current algebra $\CG$ has a natural grading that arises from the action of the scaling operator of $\Gaff$. We shall be interested in a  category of graded representations of $\CG$; from this point of view $\K$ is a more canonical object as it contains both $\CG$ and the scaling operator and a $\K$-representation is essentially a graded $\CG$-representation. The relevant category for us is denoted by ${}_{\K}\M$ and has as objects the representations of $\K$ with finite-dimensional eigenspaces for the scaling operator. For instance, the $\K$-stable Demazure modules of the highest weight $\Gaff$-representations are objects in this category, making it clear that ${}_{\K}\M$ is not a semi-simple category. From this point of view, the specialized symmetric Macdonald polynomials  are graded characters of certain universal objects in ${}_{\K}\M$, called local Weyl modules (Theorem \ref{P8ii}). 

In the light of this representation theoretical interpretation, our goal is to establish for ${}_{\K}\M$  an analog of the reciprocity result  proved by  Bernstein-Gelfand-Gelfand \cite{BGG} for the category $\O$ of a simple complex Lie algebra $L$. The reciprocity was conjectured in \cite{BCM} for current algebras associated to untwisted affine Lie algebras and it was proved there for the current algebra of type  $A_1^{(1)}$ and later in \cite{BBCKL} for the current algebra of type $A_n^{(1)}$. In this paper, we give a uniform proof of this conjecture for \emph{all} indecomposable affine Lie algebras $\Gaff$. Our argument rests on the observation in \cite{BBCKL}*{\S 3.10} that the BGG reciprocity statement for the category ${}_{\K}\M$ follows from the graded character equality of the objects that appear in the statement. In  \cite{BBCKL}, the necessary graded character equality was established using a Cauchy kernel expansion that does not exist beyond type $A$. To address this impediment, we introduce a new scalar product \eqref{qscalarproduct} on a relevant ring of characters and show that the specialized symmetric Macdonald polynomials normalized by their square norms with respect to this scalar product are the graded characters of a second family of universal objects in ${}_{\K}\M$, called global Weyl modules (Proposition \ref{P8}). Using these facts, the BGG reciprocity statement ultimately becomes the manifestation of a simple scalar product identity for specialized Macdonald polynomials (Proposition \ref{P6}). 

We present below a brief account of the BGG reciprocity statement and its proof.

Recall that the BGG reciprocity for the category $\O$ of $L$ describes the relationship between the following objects in $\O$: the simple objects $V(\l)$ (parametrized by elements of the dual of a fixed Cartan subalgebra of $L$),  their projective covers $P(\l)$ (which are indecomposable objects in $\O$),  and a third family $M(\l)$ of indecomposable objects,  called Verma modules. The Verma modules are crucial for our understanding of the category $\O$ but the categorical properties that distinguish them seem to be elusive. The projective covers of simple objects have a finite filtration by Verma modules and the Verma modules have finite length and therefore admit a Jordan-H\"older series. The BGG reciprocity states that the filtration multiplicity of a  Verma module $M(\mu)$  in the projective cover $P(\l)$ is equal to the  Jordan-H\"older multiplicity of $V(\l)$ in $M(\mu)$. 
  
The simple objects  of ${}_{\K}\M$ are parametrized by pairs $(\l,k)$ where $\l$ varies over an index set of irreducible finite-dimensional representations of the Levi factor $\G$ of $\K$ (a simple Lie algebra) and $k$ varies over the integers. For the moment, to avoid unravelling more structure and introducing a more notation, we will denote them by $V(\l,k)$ and denote their projective cover by $P(\l,k)$. The analogues  in  ${}_{\K}\M$ of the Verma modules are called a global Weyl modules and denoted by $W(\l,k)$. This concept  was defined in  \cite{ChaPre} for untwisted affine Lie algebras; we give here a uniform construction for all affine Lie algebras.  A second family of indecomposable objects in ${}_{\K}\M$ consists of the unique maximal finite-dimensional quotients of $W(\l,k)$ in ${}_{\K}\M$; these are denoted by $\Wloc(\l,k)$ and called local Weyl modules. A deep result in the theory is the freeness of $W(\l,k)$ as a module over a certain polynomial algebra $\A_\l$. This allows us, for example, to relate in a precise way the graded $\G$-characters of $W(\l,r)$ and $\Wloc(\l,r)$. 

There are similarities but also significant differences, not only between the categories $\O$ and ${}_\K\M$, but also between the structure of the Verma modules and that of the global Weyl modules. One particular complication is the fact that global Weyl modules are not of finite length. Nevertheless, there exists a BGG-type reciprocity in this context (Theorem \ref{T6}): Any $P(\l,k)$ has a filtration by global Weyl modules, and the filtration multiplicity of $W(\mu,s)$ in $P(\l, k)$ is equal to the Jordan--H\"older multiplicity of $V(\l,k)$ in $\Wloc(\mu,s)$. 

The proof follows the strategy laid out in \cite{BBCKL}. For any object $M$ in ${}_{\K}\M$ one can construct a more or less canonical descending filtration and a direct sum of global Weyl modules that surjects onto $\gr M$, the graded object associated to the filtration. To show that this construction, when applied to $P(\l,k)$, produces a bijection between the appropriate direct sum of global Weyl modules and $\gr P(\l,k)$ it is enough to compare their graded $\G$-characters and it can be seen that, in fact,  the equality of multiplicities in the BGG reciprocity statement implies that $P(\l,k)$ has a filtration by global Weyl modules. Taking into account the fact that the graded $\G$-characters of  local and global Weyl modules can be expressed in terms of specialized symmetric Macdonald polynomials, we  establish  the equality of multiplicities in the BGG reciprocity statement  as a consequence of a scalar product identity for specialized Macdonald polynomials.

In conclusion, we point out that the conventional definition of the current algebra is as follows. The current algebra $L[t]$ associated to a simple complex Lie algebra $L$ is  the Lie algebra of polynomial maps $\C\to L$. It can be identified with the $\C$-vector space $L\otimes_\C \C[t]$ with Lie bracket the $\C[t]$-bilinear extension of the Lie bracket of $L$. The current algebras $L[t]$ are isomorphic (as Lie algebras) with the algebras $\CG$ for $\Gaff$ untwisted indecomposable affine Lie algebra. A twisted current algebra is conventionally defined as a fixed point subalgebra of  $L[t]$ under an automorphism induced by a non-trivial outer automorphism of $L$. With the exception mentioned below, the twisted current algebras are isomorphic to the algebras $\CG$ for $\Gaff$ twisted indecomposable affine Lie algebra. For $L$ of type $A_{2n}$, the fixed point construction corresponds to a special, but not hyperspecial, maximal parabolic subalgebra of $\Gaff$ of type $A_{2n}^{(2)}$ and it is hence different from our algebra $\CG$. 
%The essential freeness property of the global Weyl modules is false for this non-hyperspecial maximal parabolic (see \cite{FouKus}) and this mirrors in some sense the more complicated structure of the category of smooth representations of a reductive $p$-adic group with vectors fixed by a special non-hyperspecial maximal compact open subgroup. 
The theory of local and global Weyl modules for the hyperspecial maximal parabolic in type $A_{2n}^{(2)}$ is developed in \cite{CIK}. The 
%failure of the freeness property for the special non-hyperspecial maximal parabolic in type $A_{2n}^{(2)}$ (which is obtained as a fixed point subalgebra) and its 
validity of the freeness property of the global Weyl modules for the hyperspecial maximal parabolic subalgebra indicates that perhaps there must be a more general  and coherent theory of local and global Weyl modules extending beyond the usual map algebras or equivariant map algebras studied, for example, in \cites{CFK, NSS, FMSa}.

\begin{ack} 
Part of this work was carried out during the Spring 2013 semester program ``Automorphic Forms, Combinatorial  Representation Theory and Multiple Dirichlet Series'' at the Institute for Computational and Experimental Research in Mathematics, Providence, RI. We thank ICERM and the organizers of the semeseter program for the hospitality and support. V.C. was partially supported by DMS-0901253. B.I. was partially supported by  CNCS-UEFISCDI project number 88/05.10.2011.
\end{ack}

%%%%%%%%%%%%%%%%%%%%%%%%%%%%%%%%%%%%%%%%%%%%%%%%%%%%%%%%%%
\section{Affine Lie algebras }

\subsection{} Hereafter, unless otherwise specified, all vector spaces are complex vector spaces and $\otimes$ stands for $\otimes_\C$.

\subsection{} We refer to \cite{KacInf} for the general theory of affine Lie algebras. Let $A=(a_{ij})_{0\leq i,j\leq n}$ be an indecomposable affine Cartan matrix, $S(A)$ the Dynkin diagram, and  $(a_0,\dots, a_n)$  the numerical labels of $S(A)$ in Table Aff from \cite{KacInf}*{pg.54-55}. We denote by $(a_0^\vee,\dots, a_n^\vee)$ the labels of the dual Dynkin diagram $S({}^tA)$ which is obtained from $S(A)$ by reversing the direction of all arrows and keeping the same enumeration of the vertices. The associated finite Cartan matrix is $\Aring = (a_{ij})_{1\leq i,j\leq n}$.   Note that  $a_0^\vee=1$ for all indecomposable affine Cartan matrices while $a_0=1$ in all cases except for $A=A_{2n}^{(2)}$ for which $a_0=2$.

%%%%%%%%%%%%%%%%%%%%%%%%%%%%%%%%%%%%%%%%%%%%%%%%%%%%%%%%%%
\subsection{} Let $(\Haff, R, R^{\vee})$ and $(\H, \Rring, \Rring^{\vee})$ be realizations of $A$ and  $\Aring$, respectively,  and let  $\Gaff$ and $\G$ be the associated affine and  finite--dimensional simple Lie algebras, respectively. We can arrange that $\H\subset \Haff$,  $\Rring\subset R$ and that $\G$ is a Lie subalgebra of $\Gaff$. The subspaces  $\Haff$ and  $\H$  are  the corresponding Cartan subalgebras, and $R$, $\Rring$ are the root systems corresponding to $(\Gaff, \Haff)$ and $(\G,\H)$, respectively, and we have
\begin{equation}
\Gaff=\Haff\oplus\bigoplus_{\a\in R}\Gaff_\a,\quad \G=\H\oplus\bigoplus_{\a\in\Rring}\G_\alpha .
\end{equation}
We refer to \cite{KacInf} for the details of this construction.

%%%%%%%%%%%%%%%%%%%%%%%%%%%%%%%%%%%%%%%%%%%%%%%%%%%%%%%%%%
\subsection{}

Fix a Borel subalgebra $\Haff\subset\Baff$ of $\Gaff$ and the corresponding Borel subalgebra $\H\subset\B$ of $\G$. For later use, we denote the corresponding nilpotent radicals by $\Naff$ and $\N$, respectively. Let $R^+$ and $\Rring^+$ be the set of roots of $(\Baff,\Haff)$ and $(\B,\H)$, respectively. With the notation $R^-:=-R^+$ and $\Rring^-=-\Rring^+$ we have
\begin{equation}
R=R^+\cup R^-,\quad \Rring=\Rring^+\cup\Rring^- .
\end{equation}
Let  $\{\a_i\}_{0\leq i\leq n}$ the basis of $R$ determined by $R^+$ and let  $\{\a_i^\vee\}_{0\leq i\leq n}\subset\Haff$ be the corresponding set of coroots. The  basis of $\Rring$ determined by $\Rring^+$ is  $\{\a_i\}_{1\leq i\leq n}$ and  $\{\a_i^\vee\}_{1\leq i\leq n}\subset\H$ are the associated coroots.

%%%%%%%%%%%%%%%%%%%%%%%%%%%%%%%%%%%%%%%%%%%%%%%%%%%%%%%%%%
\subsection{}

The center of $\Gaff$ is one-dimensional and  is spanned by the canonical central element
\begin{equation}
K=a_0^\vee\a_0^\vee+\cdots+a_n^\vee\a_n^\vee\in\Haff.
\end{equation}
Fix $d\in\Haff$ such that $\a_0(d)=1$ and $\a_i(d)=0$ for $1\leq i\leq n$; $d$ is called the scaling element and  is unique  modulo the subspace spanned by $K$. With this notation, we have the following decomposition
\begin{equation}
\Haff=\H\oplus\C K\oplus\C d.
\end{equation}
Let
\begin{equation}
\d=a_0\a_0+\cdots+a_n\a_n\in R^+,
\end{equation}
be the positive non-divisible null-root in $R$.  For $1\le i\leq n$, define  $\L_i\in\H^*$ by $\L_i(\a_i^\vee)=\d_{i,j}$, for $1\leq j\leq n$, where $\d_{i,j}$  is Kronecker's delta symbol. The element $\L_i$ is the fundamental weight of $\G$ corresponding to $\a_i^\vee$. We also define $\L_0\in \Haff^*$ by $\L_0(\a_i^\vee)=\delta_{0,i}$, for $0\leq j\leq n$, and $\L_0(d)=0$. The element $\L_0$ is the fundamental weight of $\Gaff$ corresponding to $\a_0^\vee$.

We have
\begin{equation}
\Haff^*=\H^*\oplus\C \d\oplus\C \L_0.
\end{equation}
 An important role is played by the root
\begin{equation}
\th=a_1\a_1+\cdots+a_n\a_n\in\Rring^+.
\end{equation}
This is  the highest root of $\Rring$ if $\Gaff$ is untwisted or of type $A_{2n}^{(2)}$, and is  the dominant short root otherwise.

%%%%%%%%%%%%%%%%%%%%%%%%%%%%%%%%%%%%%%%%%%%%%%%%%%%%%%%%%%
\subsection{} We will add $\Re$ as a subscript whenever we refer to the real form of  $\Haff$ or $\Haff^*$ spanned by the simple coroots and $d$, or by the simple roots and $\L_0$, respectively. A similar convention applies for $\H$ and $\H^*$.

The following defines the non-degenerate normalized standard bilinear  form $(\ , \ )$ on $\Haff_\Re^*$:
\begin{equation}
(\a_i,\a_j):=d_i^{-1}a_{ij}, \ 0\leq i,j\leq n\ ,\quad
(\L_0,\a_i):=\d_{i,0}a_0^{-1},\quad \text{and}\quad(\L_0,\L_0):=0,
\end{equation}
with $d_i:= a_ia_i^{{\vee}-1}$. In particular, we have
\begin{equation}
(\d,\H_\Re^*)=0, \quad (\d,\d)=0, \quad\text{and}\quad (\d,\L_0)=1
\end{equation}
The corresponding isomorphism $\nu:\Haff_\Re\to\Haff_\Re^*$ sends $\a_i^\vee$ to $d_i\a_i$, $K$ to $\d$, and $d$ to $a_0\L_0$. We will routinely identify elements via $\nu$ and regard for example, co-roots and co-weights as elements of $\Haff_\Re^*$.

%%%%%%%%%%%%%%%%%%%%%%%%%%%%%%%%%%%%%%%%%%%%%%%%%%%%%%%%%%
\subsection{}

With respect to $(\ ,\ )$, the real roots of $R$  have three possible lengths if $\Gaff$ is of type $A^{(2)}_{2n}$, $n\geq 2$, the same length if the affine Dynkin diagram is simply laced, and two possible lengths otherwise. We denote the set of short roots by $R_s$, the set of long roots by $R_\ell$, and, for $\Gaff$ of type $A^{(2)}_{2n}$, denote the set of medium length roots by $R_m$. To avoid making the distinction later on, if there is only one root length we consider all real roots to be short. Similar notation and conventions apply to  $\Rring$.

Observe  that
\begin{equation}
(\th,\th)=2a_0\quad\text{and}\quad \max_{\a\in R}(\a,\a)=2r
\end{equation}
where $r$ is the numerical label of Table Aff $r$ in \cite{KacInf} corresponding to the Dynkin diagram of $\Gaff$. For any $\alpha\in R$, it is convenient to consider the integer
\begin{equation}
r_\a:=\max\{\frac{(\a,\a)}{2}, 1\}.
\end{equation}

The imaginary roots of $R$ are $$R^{im}=(\Z\setminus\{0\})\d,$$ while the real roots are
\begin{align*}
&R^{re}:=\frac{1}{2}(\Rring+(2\Z+1)\d)\cup(\Rring +2\Z\d)  & &\text{if } A=A^{(2)}_{2},\\
&R^{re}:=\frac{1}{2}(\Rring_\ell+(2\Z+1)\d)\cup(\Rring_s+\Z\d)\cup(\Rring_\ell +2\Z\d)& &\text{if } A=A^{(2)}_{2n}, n\geq 2,\\
&R^{re}:=(\Rring_s+\Z\d)\cup(\Rring_\ell +r\Z\d)& &\text{otherwise}.
\end{align*}
The positive affine roots can be described as $R^+=R^{re,+}\cup R^{im,+}$  where $R^{im,+}=\Z_{>0}\d $,  and

\begin{align*}
R^{re,+}=&\frac{1}{2}(\Rring^++(2\Z_{\geq 0}+1)\d) \cup\frac{1}{2}(\Rring^-+(2\Z_{\geq 0}+1)\d)\cup(\Rring^+ +2\Z_{\geq 0}\d)\cup(\Rring^- +2\Z_{>0}\d)  & &\text{if } A=A^{(2)}_{2},\\
R^{re,+}=&\frac{1}{2}(\Rring^+_\ell+(2\Z_{\geq 0}+1)\d)\cup\frac{1}{2}(\Rring^-_\ell+(2\Z_{\geq 0}+1)\d)\cup(\Rring^+_s+\Z_{\geq 0}\d)\cup(\Rring^-_s+\Z_{>0}\d)& &\text{if } A=A^{(2)}_{2n}, n\geq 2,\\
&  \cup(\Rring^+_\ell +2\Z_{\geq 0}\d)\cup(\Rring^-_\ell +2\Z_{>0}\d)& &\\
R^{re,+}=&(\Rring^+_s+\Z_{\geq 0}\d)\cup(\Rring^-_s+\Z_{> 0}\d)\cup(\Rring^+_\ell +r\Z_{\geq 0}\d)\cup(\Rring^-_\ell +r\Z_{>0}\d)& &\text{otherwise}.
\end{align*}
%%%%%%%%%%%%%%%%%%%%%%%%%%%%%%%%%%%%%%%%%%%%%%%%%%%%%%%%%%
%\subsection{}

Let $\{\L_i\}_{1\leq i\leq n}$ and  $\{\L_i^\vee~|~\L_i^\vee=d_i^{-1}\L_i,~1\leq i\leq n\}$ be the fundamental weights and fundamental coweights of $\G$, respectively. The weight and root lattices of $\G$  are denoted by $\P$ and $\Q$ and the  lattice spanned by $\{\a_i^\vee~|~ 1\leq i\leq n\}$, is denoted by $\Q^\vee$. The root  lattice of $\Gaff$ is denoted by $Q$. For us, the relevant set of weights of $\Gaff$ is the integral lattice spanned by $\{\L_i\}_{0\leq i\leq n}$  and $a_0^{-1}\d$, which we denote by $P$. We have $P=\P\oplus \Z\L_0\oplus \Z a_0^{-1}\d$, while $Q=\Q\oplus\Z\d$ unless $\Gaff$ is of type $A^{(2)}_{2n}$ in which case $\Q\oplus\Z\d\subset Q$ is a sub-lattice of  index two. To be able to have uniform statements in certain situations we need to consider the lattice $\Qbar\subset \H^*_\Re$, defined as the orthogonal projection of $Q\subset \Haff^*_\Re$ onto $\H^*_\Re$. The lattice $\Qbar$ equals $\Q$ unless $\Gaff$ is of type $A^{(2)}_{2n}$ for which $\Q\subset \Qbar$ is a sub-lattice of index two. The set of dominant and anti-dominant weights of $\G$ is denoted by $\P^+$ and $\P^-$, respectively; the set of dominant elements of $P$ of $\Gaff$  is denoted by $P^+$. Define $\Q^+$ and $\Qbar^+$ as the cone spanned by $\Rring^+$ inside $\Q$ and, respectively, $\Qbar$.
%%%%%%%%%%%%%%%%%%%%%%%%%%%%%%%%%%%%%%%%%%%%%%%%%%%%%%%%%%
\subsection{}

Given $\a\in R^{re}$ and $x\in \Haff^*$ let
\begin{equation}
s_\a(x):=x-\frac{2(x,\a)}{(\a,\a)}\a\ .
\end{equation}
The affine Weyl group $W$ is the subgroup of ${\GL}(\Haff_\Re^*)$ generated by all $s_\a$ (the simple reflections $s_i=s_{\a_i}$, $0\le i\leq n$ are enough). The finite Weyl group $\W$ is the subgroup generated by $s_1,\dots,s_n$. The bilinear form on $\Haff_\Re^*$ is equivariant with respect to the affine Weyl group action. Both the finite and the affine Weyl group are Coxeter groups and they can be abstractly defined as generated by $s_1,\dots,s_n$, respectively $s_0,\dots,s_n$, and the following relations:
\begin{enumerate}[label={\alph*)}]
\item reflection relations: $s_i^2=1$;
\item braid relations: $s_is_j\cdots =s_js_i\cdots $ (there are $m_{ij}$ factors on each side, $m_{ij}$ being equal to $2,3,4,6$ if the number of laces connecting the corresponding nodes in the Dynkin diagram is $0,1,2,3$ respectively).
\end{enumerate}
For $x\in \Haff^*$ we denote by $W(x)$ and $\W(x)$ the orbit of $x$ under the action of $W$ and $\W$, respectively.

Let $M\subset\H_\Re^*$ be the lattice generated by $\W(a_0^{-1}\theta)$.  The affine Weyl group contains the finite Weyl group and a normal abelian subgroup isomorphic to $M$. We will denote the latter by $\tau(M)$ and its elements by $\tau_\mu$, $\mu\in M$. The action by conjugation of $\W$ on $\tau(M)$ and the usual action of $\W$ on $M\subset \H_\Re^*$ are related by
\begin{equation}
\w\tau_\mu \w^{-1}=\tau_{\w(\mu)}.
\end{equation}
This allows $W$ to be presented as a semidirect product $W\cong\W\ltimes M$.

For any real number $s$, the subset  $\Haff^*_s=\{ x\in\Haff_\Re\ ~|~ (x,\d)=s\}$ is an (affine) hyperplane of $\Haff_\Re^*$, called the level $s$ of $\Haff_\Re^*$. We have
\begin{equation}
\Haff^*_s=\Haff^*_0+s\L_0=\H_\Re^*+{\mathbb R}\d+s\L_0\ .
\end{equation}
The action of $W$ preserves each $\Haff^*_s$ and  we can identify each level canonically with $\Haff^*_0$ and obtain an (affine) action of $W$ on $\Haff^*_0$. If $s_i\in W$ is a simple reflection, write $s_i(\cdot)$ for the usual (level zero) action of $s_i$ on $\Haff^*_0$ and $s_i\<\cdot\>$ for the affine action of $s_i$ on $\Haff^*_0$ corresponding to the level one action. For example, the level zero action of $s_0$ and $\tau_\mu$ is
\begin{equation}
\begin{aligned}
s_0(x)     & =  s_\th(x)+(x,\th)a_0^{-1}\d\ ,\\
\tau_\mu(x)   & =  x - (x,\mu)\d\ ,
\end{aligned}
\end{equation}
and the level one action of the same elements is
\begin{equation}
\begin{aligned}
s_0\<x\>   & =  s_\th(x)+(x,\th)a_0^{-1}\d-\a_0\ ,\\
\tau_\mu\<x\> & =  x + \mu - (x,\mu)\d -\frac{1}{2}|\mu|^2\d \ .
\end{aligned}
\end{equation}
The level one action on $\Haff^*_0$ induces an affine action of $W$ on $\H_\Re^*$. As a matter of notation, we write $w\cdot x$
for the level one affine action of $w\in W$ on $x\in \H_\Re^*$. For example,
\begin{equation}
\begin{aligned}
s_0\cdot x   & =  s_\th(x)+a_0^{-1}\th\ ,\\
\tau_\mu\cdot x & =  x  + \mu  \ .
\end{aligned}
\end{equation}
The fundamental alcove is defined as
\begin{equation}
{\mathcal C}:=\{ x\in \H_\Re^*\ | \ (x+\L_0,\a_i^\vee)\geq 0\ ,\ 0\leq
i\leq n\}\ .
\end{equation}
We remark that 
\begin{equation}
P^+_1:=\left(\P\cap{\mathcal C}\right)+\L_0
\end{equation} 
is the set of   level one dominant weights of $\Gaff$ and that
\begin{equation}
\O_{\P}:=\P\cap {\mathcal C}
\end{equation}
is a set of representatives for the level one $W$ orbits on $\P$.

If we examine the orbits of the level zero action of the affine Weyl group $W$  on the real affine roots $R^{re}$ we find the following:
\begin{enumerate}[label={\alph*)}]
\item if $\Gaff$ is not of type $C_n^{(1)}$ there are as many orbits as root lengths;
\item if $\Gaff$ is of type $C_n^{(1)}$ then there are three orbits:
$$ W(\a_1)=\Rring_s+\Z\d, \quad W(\a_n)=\Rring_\ell+2\Z\d\quad\text{and}\quad W(\a_0)=\Rring_\ell+(2\Z+1)\d.$$
\end{enumerate}

%%%%%%%%%%%%%%%%%%%%%%%%%%%%%%%%%%%%%%%%%%%%%%%%%%%%%%%%%%

\subsection{}
For $\L\in P^+$, let $V(\L)$ be  the unique irreducible  highest weight  $\Gaff$-module with highest weight $\L$. For each $w\in W$ and $\L\in P^+$ the weight space $V(\L)_{w(\L)}$ is one-dimensional. The $\Baff$-module generated by $V(\L)_{w(\L)}$  depends only on  $w(\L)$ and is denoted by $D(w(\Lambda))$.  These modules  are  called  Demazure modules and they are finite dimensional vector spaces. The eigenspaces of the scaling element $d$ define a canonical $\Z$-grading  on the Demazure modules. Replacing $\L$ with $\L+sa_0^{-1}\d$ produces modules  that differ only in the $d$ action: the grading of one is a shift  of the other.
In this paper, we are concerned  with the Demazure modules of $\Gaff$ associated to a level one dominant weight. Since
\begin{equation}
P^+_1=(\P\cap\mathcal C)+\L_0,
\end{equation}
the level one Demazure modules are of the form $D(\l+sa_0^{-1}\d+\L_0)$ with $\l\in\P$. As these are the only Demazure modules we will be concerned with, to keep the notation as simple as possible, we will use $D(\l+sa_0^{-1}\d)$ to refer to $D(\l+sa_0^{-1}\d+\L_0)$.

To a Demazure module $D(w(\L))$ we can associate its character
\begin{equation}\label{Demchr}
\chi(D(w(\L)))=e^{-\L_0}\sum_{\Upsilon\in P} \dim_\C(D(w(\L))_{\Upsilon})
\cdot e^{\Upsilon}.
\end{equation}

For $\l\in \P^+$,  let $\V(\l)$ be the unique highest weight  $\G$-module with highest weight $\l$. The  Demazure modules in $\V(\l)$ are defined in an analogous way and will be denoted by $\D(\mu)$, where $\mu\in\W(\l)$. If $w_\circ$ denotes the longest element of $\W$, then  $\D(w_\circ(\l))$ coincides with the highest weight module $\V(\l)$.

%%%%%%%%%%%%%%%%%%%%%%%%%%%%%%%%%%%%%%%%%%%%%%%%%%%%%%%%%
\subsection{}\label{convention}  We shall say that  $\Gaff$ is an affine Lie algebra of type I if it is indecomposable and either a simply-laced untwisted affine Lie algebra or a twisted affine Lie algebra. We shall say that $\Gaff$ is an affine Lie algebra of type II if it is an indecomposable, non-simply-laced, untwisted affine Lie algebra.

  The type I algebras are distinguished by certain special properties and we list those which are most relevant for our study. They are precisely the indecomposable affine Lie algebras for which   $\a_0$ is a short root. Further, $M=\P$ if  $\Gaff$ is of type $A_{2n}^{(2)}$ and $M=\Q$ otherwise. The set of  non-zero elements of $\O_{\P}$ is empty if $\Gaff$ is of type $A_{2n}^{(2)}$ and is the set of minuscule weights of $\G$ otherwise.
  
The root system $\Rring$ determines  the affine Lie algebra  $\Gaff$ of type I  uniquely except  in the case when    of type $A_{2n-1}^{(2)}$ and $A_{2n}^{(2)}$
 when $\Rring$ is of type $C_n$. A faithful invariant in this sense is
\begin{equation}
\Rbar=\{\a_{|\H}~|~\a\in R\}\subset \H^*
\end{equation}
which is again a  root system, possibly non-reduced. We have $\Rbar=\Rring$ in all cases except for $\Gaff$ of type $A_{2n}^{(2)}$ for which $\Rbar$ is $BC_n$, the unique irreducible non-reduced root system of rank $n$. The correspondence between the set of type I affine Lie algebras $\Gaff$ and the set of irreducible (possibly non-reduced) root systems $\Rbar$ is bijective. The lattice $\Qbar$ is precisely the lattice spanned by $\Rbar$.

%%%%%%%%%%%%%%%%%%%%%%%%%%%%%%%%%%%%%%%%%%%%%%%%%%%%%%%%%%

%%%%%%%%%%%%%%%%%%%%%%%%%%%%%%%%%%%%%%%%%%%%%%%%%%%%%%%%%%
\section{Current algebras}

\subsection{}\label{currentdef} An important Lie algebra is the maximal standard parabolic subalgebra of $\Gaff$ corresponding to the Dynkin diagram of $\G$. The center of $\Gaff$ splits over this maximal parabolic and for this reason it is preferable to work with the maximal parabolic modulo the center. Therefore, let
\begin{equation}
\K=(\H\oplus\C d)\oplus\bigoplus_{\a\in R^+}\Gaff_\a\oplus\bigoplus_{\a\in\Rring^-}\Gaff_\a .
\end{equation}
Let us also establish the following notation: $R_{\K}=R^+\cup \Rring^-$, $R_{\K}^-=R_{\K}\setminus R_{\K}^+$, $R_{\K}^+=R_{\K}^{re,+}\cup R^{im,+}$, and
\begin{align*}
R_\K^{re,+}=&\left(\frac{1}{2}(\Rring^++(2\Z+1)\d)\cup(\Rring^++\Z\d)\right)\cap R^{re,+}  & &\text{if } A=A^{(2)}_{2},\\
R_\K^{re,+}=&\left(\frac{1}{2}(\Rring^+_\ell+(2\Z+1)\d)\cup(\Rring^++\Z\d)\right)\cap R^{re,+} & &\text{if } A=A^{(2)}_{2n}, n\geq 2,\\
R_\K^{re,+}=&(\Rring^++\Z\d)\cap R^{re,+}& &\text{otherwise}.
\end{align*}
Define
\begin{equation}
 \Naff^{re}_\K=\bigoplus_{\a\in R_\K^{re,+}}\Gaff_\a\subset \Naff_\K=\bigoplus_{\a\in R_\K^+}\Gaff_\a\supset \Naff^{im}=\bigoplus_{k\in\Z_{>0}}\Gaff_{k\d},\quad \bar\Naff_\K=\bigoplus_{\a\in R_\K^{-}}\Gaff_\a
\end{equation}

The current algebra is  defined as the ideal of $\K$ described as
\begin{equation}
\CG=\H\oplus\bigoplus_{\a\in R^+}\Gaff_\a\oplus\bigoplus_{\a\in\Rring^-}\Gaff_\a
\end{equation}
Note that $\Naff_\K=\Naff_\K^{re}\oplus\Naff^{im}$ and $\CG=\bar\Naff_\K\oplus\H\oplus\Naff^{im}\oplus\Naff^{re}_\K$. The scaling element $d$ acts both on $\K$ and $\CG$ inducing Lie algebra $\Z_{\geq 0}$-gradings. The degree zero Lie subalgebras are $\K(0)=\G\oplus\C d$ and $\CG(0)=\G$, respectively. The subspaces consisting of strictly positive homogeneous components are ideals denoted by $\K_+$ and $\CG_+$, respectively. Therefore, the following short exact sequences split as $\K$--modules:
\begin{equation}
0\to \K_+\to\K\to \K(0)\to 0, \quad 0\to \CG_+\to\CG\to \CG(0)\to 0
\end{equation}
For later use we denote by $\pi:\K\to \K(0)$ the canonical surjection and by $\vartheta:\K(0)\to \K$ the canonical injection (a splitting of $\pi$); they are Lie algebra morphisms.

\begin{Rem} The current algebra is typically defined in the literature as follows. Let $L$ denote a simple Lie algebra and $\sigma$ a diagram automorphism of $L$ whose order we denote by $m$. Recall that there is a bijective correspondence between (isomorphism classes of) indecomposable affine Lie algebras $\Gaff$ and pairs $(L,\sigma)$; under this correspondence $\G$ is isomorphic to the fixed point Lie subalgebra $L^\sigma$. In fact, the Kac label for the Dynkin diagram of $\Gaff$ is $X^{(m)}$ where $X$ is the Dynkin type of $L$. Given  an indeterminate $t$, the  Lie algebra $L[t]$ is defined as  the vector space $L\otimes \C[t]$ with Lie bracket given by $[x\otimes f, y\otimes g]=[x,y]\otimes fg$ for $x,y\in L$ and $f,g\in\C[t]$. The automorphism $\sigma$ can be extended to $L[t]$ by  $\sigma(x\otimes f(t)):=\sigma(x)\otimes f(e^{-2\pi i/m} t)$. The fixed point Lie subalgebra $L[t]^\sigma$ is called a current algebra in the literature. The Lie algebras $\CG$ and $L[t]^\sigma$ are isomorphic in all cases with the exception of $\Gaff$ of type $A_{2n}^{(2)}$. In this situation, $L[t]^\sigma$  corresponds to the special not hyperspecial standard maximal parabolic subalgebra of $\Gaff$ and $\CG$ corresponds to the hyperspecial standard maximal parabolic subalgebra $\K$ of $\Gaff$ (see \cite{TitRed}). Therefore, for $\Gaff$ of type $A_{2n}^{(2)}$ all the concepts discussed in this section are considered here for the first time.
\end{Rem}
%%%%%%%%%%%%%%%%%%%%%%%%%%%%%%%%%%%%%%%%%%%%%%%%%%%%%%%%%%
\subsection{}

Let $M$ be $\K$-module. Since the central element is not contained in $\K$ the relevant set of integral weights is $P_\K=\P\oplus a_0^{-1}\Z\d$. We denote by $P_\K^+$ the subset of dominant elements of $P_\K$; more precisely, $P_\K^+=\P^+\oplus a_0^{-1}\Z\d$. We say that $M$ is a weight $\K$-module if
\begin{equation}
M=\bigoplus_{\L\in P_\K}M_\L ,
\end{equation}
where $M_\L=\{m\in M~|~h\cdot m=\L(h)m, \forall h\in \H\oplus\C d\}$ is the weight space of $M$ corresponding to $\L$. Denote by $\wtaff(M)$ the set of weights for which the weight spaces are non-zero. Note that the weight $\K$-modules have a canonical $\Z$-grading given by the  eigenspaces of $d$.

We have a corresponding notion of weight $\CG$-module where the set indexing the weights is $\P$ and the abelian subalgebra is $\H$; in other words, a weight  $\G$-module. The set of weights with non-zero weight spaces for the weight $\CG$-module $M$ is denoted by $\wt(M)$. We use the same notation for any weight $\G$-module.

As it is well-known, for $\l\in \P^+$ we have $\wt(\V(\l))=\W(\{\mu\in \P^+~|~ \l-\mu\in \Q^+\})$. Again, to be able to write some uniform statements we also need to consider the set
$$
\wtbar(\V(\l))=\W(\{\mu\in \P^+~|~ \l-\mu\in \Qbar^+\})
$$
which is identical with $\wt(\V(\l))$, unless $\Gaff$ is of type $A_{2n}^{(2)}$.
%%%%%%%%%%%%%%%%%%%%%%%%%%%%%%%%%%%%%%%%%%%%%%%%%%%%%%%%%%
\subsection{}

We define the following two categories. The first category, denoted by ${}_{\CG}\M^\Z$, has as objects the  $\Z$-graded weight $\CG$-modules with finite--dimensional graded components; the morphisms are  maps of $\Z$-graded $\CG$-modules. The second category, denoted by ${}_{\K}\M$, has as objects the  weight $\K$-modules such that the eigenspaces of $d$ are finite dimensional; the morphisms are maps of  $\K$-modules. For current algebras, the homological properties of the category ${}_{\CG}\M^\Z$ were first studied in \cite{ChaGre}, and the interest in these properties, gradually  emerged from a long sequence of developments that started with \cites{ChaInt, ChaPre} as a category relevant for the structure of the category of  finite dimensional representations of (quantum) affine Lie algebras; it is also the natural context for certain questions in mathematical physics.

Let $*: {}_{\K}\M\to{}_{\K}\M$ be the following duality functor. If $M$ is an object in ${}_{\K}\M$ then $M^*$ is the $\K$-submodule of $\Hom_\C(M,\C)$ which is spanned by the linear duals of the $d$-eigenspaces of $M$; for morphisms $f:M\to N$, $f^*:N^*\to M^*$ is the restriction to $N^*$ of $\Hom_\C(\cdot,\C)(f)$. It is a straightforward check to see that
$*^2$ is naturally isomorphic to the identity functor of ${}_{\K}\M$.

\begin{Prop}\label{P1}
${}_{\CG}\M^\Z$ and ${}_{\K}\M$ are isomorphic.
\end{Prop}
\begin{proof}
The relevant functors relating the two categories
\begin{equation}
F: {}_{\K}\M\rightarrow  {}_\CG\M^\Z, \quad  G: {}_\CG\M^\Z\rightarrow {}_{\K}\M,
\end{equation} are defined as follows.
The functor $F$  remembers only the $\CG$-module structure and the $\Z$-grading  induced by the action of $d$. The functor $G$  extends the $\CG$-module structure to a $\K$-module structure by letting $d$ act as multiplication by  $m$ on the $m$--th graded component of a $\CG$-module. It is straightforward to verify that the two functors are inverse to each other.
\end{proof}
Many  of the concepts that have been studied in the context of the current algebra and the category ${}_{\CG}\M^\Z$ can be, thanks to Proposition \ref{P1}, naturally studied in the context of the maximal parabolic algebra and the category  ${}_{\K}\M$ and we adopt this  point of view in this paper.
%%%%%%%%%%%%%%%%%%%%%%%%%%%%%%%%%%%%%%%%%%%%%%%%%%%%%%%%%%

%%%%%%%%%%%%%%%%%%%%%%%%%%%%%%%%%%%%%%%%%%%%%%%%%%%%%%%%%%
\subsection{}

Let ${}_{\K(0)}\M$ be the category with objects that are $\K(0)$-modules such that the eigenspaces of $d$ are finite dimensional; the morphisms are morphisms of $\K(0)$-modules. Another description of this category would be as the category of graded $\G$-modules with finite dimensional graded components. It is a semi-simple category and its simple objects are the irreducible highest weight modules with highest weight of the form $\l+k\d\in P_\K^+$; we denote such modules by $\V(\l+k\d)$. An alternative description of $\V(\l+k\d)$ would be as the graded $\G$-module whose unique non--zero  homogeneous component is $\V(\l)$ in degree $a_0k$. Note that $\V(\l+k\d)$ are projective objects in ${}_{\K(0)}\M$.

The Lie algebra homomorphism $\pi:\K\to \K(0)$ induces the pull-back functor
\begin{equation}
\pi^*: {}_{\K(0)}\M \to {}_{\K}\M.
\end{equation}
The map $\vartheta:\K(0)\to \K$ gives rise to  the usual induction and restriction functors
\begin{equation}
\ind_{\K(0)}^\K:  {}_{\K(0)}\M \to {}_{\K}\M, \quad  \res_{\K(0)}^\K: {}_{\K}\M \to {}_{\K(0)}\M ,
\end{equation}
where, $\ind_{\K(0)}^\K M=U(\K)\otimes_{U(\K(0))}M$. Recall that $\ind_{\K(0)}^\K$ is a left adjoint for $\res_{\K(0)}^\K$.

For all $\a\in\Rring$ let $x_{\a}\in\G_{\a}$  be such that the set  $\{x_\a~|~\a\in \Rring\}\cup\{\a_i^\vee~|~1\leq i\leq n\}$ is a Chevalley basis of $\G$.

\begin{Prop}\label{P2} Up to isomorphism, the following hold:
\begin{enumerate}[label={(\roman*)}]
\item $\pi^*\V(\l+k\d)$, $\l+k\d\in P_\K^+$, are the simple objects in ${}_{\K}\M$.
\item $\ind_{\K(0)}^{\K}\V(\l+k\d)$ is the projective cover of its unique simple quotient $\pi^*\V(\l+k\d)$.
\end{enumerate}
\end{Prop}
\begin{proof} This is part of  Proposition 1.3 and Proposition 2.1 in \cite{ChaGre}.  We briefly sketch the argument for the reader's convenience.

For the first part, we note that since $\K$ is graded by $\mathbb Z_{\ge 0}$, the action of $\K$ can only raise degree and therefore a non-homogeneous module has proper submodules. Since homogeneous objects are finite dimensional $\G$-modules  our claim follows.

For  the second part, $\ind_{\K(0)}^\K$ being a left adjoint is right exact and sends projective objects to projective objects. Therefore,  $\ind_{\K(0)}^{\K}\V(\l+k\d)$ is projective. The kernel of the epimorphism
\begin{equation}
\ind_{\K(0)}^{\K}\V(\l+k\d)\to \pi^*\V(\l+k\d)
\end{equation}
is $\K_+U(\K_+)\otimes \V(\l+k\d)$ which can be seen to be a superfluous submodule of $U(\K)\otimes_{U(\K(0))}\V(\l+k\d)$ because $1\otimes_{U(\K(0))}\V(\l+k\d)$ generates it. If
\begin{equation}
\ind_{\K(0)}^{\K}\V(\l+k\d)\to \pi^*\V(\mu+m\d),
\end{equation}
with $\mu+m\d\neq\l+k\d$ is another simple quotient then $1\otimes_{U(\K(0))}\V(\l+k\d)$ must be in the kernel so the quotient map is trivial.
\end{proof}

%%%%%%%%%%%%%%%%%%%%%%%%%%%%%%%%%%%%%%%%%%%%%%%%%%%%%%%%%%
\subsection{}
 For $\l+k\d\in P_\K^+$, let  $v_{\l+k\d}$ be a highest weight vector of $\V(\l+k\d)$. Then, the simple $\K(0)$-module $\V(\l+k\d)$ is generated by $v_{\l+k\d}$ with the relations
\begin{equation}
\begin{aligned}
&\N \cdot v_{\l+k\d}=0,\\
&h\cdot v_{\l+k\d}=(\l+k\d)(h)v_{\l+k\d},~h\in\H\oplus\C d,\\
&x_{-\a}^{(\l,\a^\vee)+1}\cdot v_{\l+k\d}=0, ~ \a\in\Rring^+.
\end{aligned}
\end{equation}

 Writing, $P(\l+k\d)=\ind_{\K(0)}^{\K}\V(\l+k\d)$ and $p_{\l+k\d}=1\otimes v_{\l+k\d}$, we observe that the  projective cover  $P(\l+k\d)$ is the $\K$-module generated by $p_{\l+k\d}$ with the same relations.

%%%%%%%%%%%%%%%%%%%%%%%%%%%%%%%%%%%%%%%%%%%%%%%%%%%%%%%%%%
\subsection{} Let  $\leq$ be the dominance partial order on $\P^+$: $\mu\le\lambda$ if and only if $\lambda-\mu\in {\Qbar^+}$. Following \cite{CFK}, we define for $\l+k\d\in P_\K^+$, the global Weyl module $W(\l+k\delta)$ as
\begin{equation}
 W(\l+k\d)=\frac{P(\l+k\d)}{\sum_{\mu\not\leq\l, s\in a_0^{-1}\Z}U(\K)P(\l+k\d)_{\mu+s\d}}.
\end{equation}

\begin{Prop}\label{P3} For $\l+k\d\in  P_\K^+$, we have
\begin{enumerate}[label={(\roman*)}]
\item $W(\l+k\d)$ is the maximal quotient of $P(\l+k\d)$ such that $\wt(W(\l+k\d))\subseteq {\wtbar}(\V(\l))$.
\item $W(\l+k\d)$ is the maximal quotient of $P(\l+k\d)$ such that $\wt(W(\l+k\d))\cap \{\mu~|~\l<\mu\}=\emptyset$.
\item $W(\l+k\d)$ is the $\K$-module generated by an element $w_{\l+k\d}$ with the relations
\begin{equation}\label{gWeylrels}
\begin{aligned}
&\Naff_\K^{re} \cdot w_{\l+k\d}=0,\\
&h\cdot w_{\l+k\d}=(\l+k\d)(h)w_{\l+k\d},~h\in\H\oplus\C d,\\
&x_{-\a}^{(\l,\a^\vee)+1}\cdot w_{\l+k\d}=0, ~ \a\in\Rring^+.
\end{aligned}
\end{equation}
\item $\dim_\C (W(\l+k\d)_{\l+k\d})=1$ and  $W(\l+k\d)$ is indecomposable with $\pi^*\V(\l+k\d)$ as its unique simple quotient.
\end{enumerate}
\end{Prop}
\begin{proof} The first three statements are consequences of the PBW theorem; see, for example, \cite{CFK}*{Proposition 4} and \cite{BBCKL}*{Proposition 3.3, Lemma 3.7} for details. Part (iii) follows from part (ii) and Proposition \ref{P2} (ii).
\end{proof}
{As a $\CG$-module, the global Weyl module $W(\l+k\d)$ is the $\CG$-module generated by the element $w_{\l+k\d}$ with the relations
\begin{equation}\label{CG-gWeylrels}
\begin{aligned}
&\Naff_\K^{re} \cdot w_{\l+k\d}=0,\\
&h\cdot w_{\l+k\d}=\l(h)w_{\l+k\d},~h\in\H,\\
&x_{-\a}^{(\l,\a^\vee)+1}\cdot w_{\l+k\d}=0, ~ \a\in\Rring^+.
\end{aligned}
\end{equation}
Remark that, as $\CG$-modules, the global Weyl modules for fixed $\l$ are isomorphic, their distinguishing property as $\K$-modules being their graded structure.}

The global Weyl modules were originally introduced by generators and relations for the untwisted affine Lie algebras in  \cite{ChaPre} but, as pointed out in \cite{ChaLok}, the definition can be made for current algebras and all the results go over to the case of the current algebras. The theory of global Weyl modules for the current algebras associated to twisted affine Lie algebras is developed in \cite{CIK}.  
%%%%%%%%%%%%%%%%%%%%%%%%%%%%%%%%%%%%%%%%%%%%%%%%%%%%%%%%%%
\subsection{} Consider the commutative algebra
\begin{equation}
\A=U(\H\oplus\Naff^{im})=U\left(\bigoplus_{k\in\Z_{\geq 0}}\Gaff_{k\d}\right).
\end{equation}
The scaling element $d$ normalizes $\H\oplus\Naff^{im}$ and therefore endows $\A$ with a canonical $\Z_{\geq 0}$-grading.

Since $\A$ is a subalgebra of $U({\CG})$, the enveloping algebra $U({\CG})$ acquires a right $\A$-module structure. Keeping in mind that  the global Weyl module is a cyclic $U({\CG})$-module, one can see that the left ideal generated by the relations \eqref{CG-gWeylrels} is also a right $\A$-submodule (see \cite{CFK}*{\S 3.4} for details). The annihilator of $w_{\l+k\d}$ in $\A$ is a graded ideal of $\A$, independent of $k$. With the notation
\begin{equation}
\A_{\l}=\A/\Ann_\A(w_{\l}),
\end{equation}
we see that the global Weyl module $W(\l+k\d)$ is a $U({\CG})-\A_\l$-bimodule. {The graded structure of $W(\l+k\d)$ is compatible with the grading of $\A$ and therefore $W(\l+k\d)$ is a $U(\K)-\A_\l$-bimodule.}

\begin{Thm}\label{T1} For $\l+k\d\in P_\K^+$, the following hold
\begin{enumerate}[label={(\roman*)}]
\item The map $$\A_\l\to \bigoplus_{j\in\Z}W(\l+k\d)_{\l+j\d}=\bigoplus_{j\geq k}W(\l+k\d)_{\l+j\d}, \  a\mapsto w_{\l+k\d}\cdot a$$ is a linear isomorphism.
\item $W(\l+k\d)$ is a finitely generated $\A_\l$-module.
\item The algebra $\A_\l$ is isomorphic as a graded algebra to the polynomial $\C$-algebra in variables $T_{i,r}$, $1\le i\le n$, $1\le r\le(\lambda,\alpha_i^\vee)$, where  $T_{i,r}$ has degree $a_0r_{\alpha_i}$ if $i<n$ and degree $r_{\alpha_n}$ if $i=n$.
\end{enumerate}
\end{Thm}
\begin{proof} Part (i) is immediate from the definition of $\A_\l$. Parts (ii) and (iii) were proved in  \cite{CFK}*{Theorem 2} and \cite{CFK}*{Theorem 4} for $\Gaff$ untwisted. The results were proved in  
\cite{CIK}*{Theorem 1, Proposition 6.3, Theorem 8} for $\Gaff$ twisted. 
\end{proof}
%%%%%%%%%%%%%%%%%%%%%%%%%%%%%%%%%%%%%%%%%%%%
\subsection{} The local Weyl modules  are defined as
\begin{equation}
\Wloc(\l+k\d,\m)= W(\l+k\d)\otimes_{\A_\l}(\A_\l/\m),
\end{equation}
 where $\m$ is a maximal ideal of $\A_\l$. Using Theorem \ref{T1}(ii), we see that the local Weyl modules are finite-dimensional.  Their  importance in the study of the  category of finite dimensional $U(\K)$-modules arises from  the fact that any finite-dimensional cyclic highest weight $U(\K)$-module generated by a one-dimensional $\l$-weight space is a quotient of $\Wloc(\l,\m)$  for some maximal ideal $\m$ of $\A$ (see \cite{ChaPre}*{Proposition 2.1} and \cite{CFS}*{Lemma 2.5}). The local Weyl module corresponding to the unique maximal homogeneous ideal of $\A_\l$ is an object of ${}_{\K}\M$; we denote this module by $\Wloc(\l+k\d)$.

\begin{Prop}\label{P4} For $\l+k\d\in P_\K^+$, the following hold
\begin{enumerate}[label={(\roman*)}]
\item  $\Wloc(\l+k\d)$ is the maximal quotient of $P(\l+k\d)$ such that $\wt(\Wloc(\l+k\d)\subseteq {\wtbar}(\V(\l))$ and $\dim_\C \Wloc(\l+k\d)_{\l+j\d})=\delta_{j,k}$ for $j\in a_0^{-1}\Z$.
\item  $\Wloc(\l+k\d)$ is a finite dimensional, indecomposable object of ${}_{\K}\M$ with unique simple quotient  $\pi^*\V(\l+k\d)$.
\item  $\Wloc(\l+k\d)^*$ has simple socle  $\soc(\Wloc(\l+k\d)^*)=(\pi^*\V(\l+k\d))^*$.
\item $\Wloc(\l+k\d)$ is the $\K$-module generated by an element $u_{\l+k\d}$ with the relations
\begin{equation}\label{lWeylrels}
\begin{aligned}
&\Naff_\K \cdot u_{\l+k\d}=0,\\
&h\cdot u_{\l+k\d}=(\l+k\d)(h)u_{\l+k\d},~h\in\H\oplus\C d,\\
&x_{-\a}^{(\l,\a^\vee)+1}\cdot u_{\l+k\d}=0, ~ \a\in\Rring^+ .
\end{aligned}
\end{equation}
%\item $\Wloc(\l+k\d)\otimes\A_\l$ and $W(\l+k\d)$ are isomorphic as graded $U(\K(0))-\A_\l$-bimodules.
\end{enumerate}
\end{Prop}
\begin{proof} The proof  is immediate from Theorem \ref{T1} and Proposition \ref{P3}.
\end{proof}
%%%%%%%%%%%%%%%%%%%%%%%%%%%%%%%%%%%%%%%%%%%%%%%%%%%%%%%%%%%%%%%%%%%%%%%%%%%%

\subsection{}  The Demazure module $D(\l+s\a_0^{-1}\d)$, which usually is only a $\Baff$-module is a $\K$--module if $\l$ is an anti-dominant weight. For $\l\in\P$ we denote by $\tilde\l$ the representative  in $\O_{\P}$ of the level one $W$ orbit of $\l$.

A fundamental result on local Weyl modules is their connection with Demazure modules. The following result was proved in \cite{ChaLok}*{Corollary 1.5.1} for $\Gaff$ of type $A_{n}^{(1)}$, in \cite{FouLit}*{Theorem A} for $\Gaff$ untwisted of type I, and in \cite{FouKus}*{Theorem 5.0.2} for $\Gaff$ twisted of type I but not of type $A_{2n}^{(2)}$ and in \cite{CIK}*{Theorem 2} for $A_{2n}^{(2)}$. %all twisted types.

\begin{Thm}\label{T2}
Let $\Gaff$ be an affine Lie algebra of type I and let $\l+k\d\in P_\K^+$. The local Weyl module $\Wloc(\l+k\d)$ and the Demazure module $D(w_\circ(\l)+k\d)$   are isomorphic as $\K$-modules.
\end{Thm}
\begin{Rem} Theorem \ref{T2} is not true if $\Gaff$ is of type II. In this case, the Demazure module is a proper quotient of the local Weyl module. Further details may be found in \cite{NaoWey}.
\end{Rem}
%%%%%%%%%%%%%%%%%%%%%%%%%%%%%%%%%%%%%%%%%%%%%%%%%%%%%%%%%%%%%%%%%
\subsection{}
The following crucial result was conjectured in \cite{ChaPre} and proved there for $\Gaff$ of type $A_{1}^{(1)}$,  in \cite{ChaLok} for $\Gaff$ of type $A_{n}^{(1)}$, in \cite{FouLit}*{Corollary B} for untwisted simply--laced affine algebras,  in \cite{NaoWey}*{Corollary A} for the untwisted non--simply laced affine Lie algebras, and in \cite{CIK}*{Theorem 3, Theorem 10} for $\Gaff$ twisted.
\begin{Thm}\label{T7}
Let $\lambda+k\d\in P_\K^+$ and let $\m$ be a maximal ideal of $\A_\l$. Then, $$\dim_\C \Wloc(\l+ k\d, \ \m)=\dim_\C\Wloc(\l+k\d).$$
 \end{Thm}
 Theorem \ref{T7} implies that the global Weyl module is a projective $\A_\l$--module. Using  Theorem \ref{T1}(iii) and the Quillen--Suslin theorem \cites{QuiPro, SusPro} we obtain the following.
 \begin{Cor}\label{P4iii} The global Weyl module $W(\l+k\d)$ is a free $\A_\l$-module of finite rank and $$\Wloc(\l+k\d)\otimes\A_\l\cong W(\l+k\d)$$  as graded $U(\K(0))-\A_\l$-bimodules.
 \end{Cor}
%%%%%%%%%%%%%%%%%%%%%%%%%%%%%%%%%%%%%%%%%%%%%%%%%%%%%%%%%%

%%%%%%%%%%%%%%%%%%%%%%%%%%%%%%%%%%%%%%%%%%%%%%%%%%%%%%%%%%

%%%%%%%%%%%%%%%%%%%%%%%%%%%%%%%%%%%%%%%%%%%%%%%%%%%%%%%%%
\subsection{}

For  an object $M$  of ${}_{\K}\M$, let $\head(M)$ be its maximal semi-simple quotient. The kernel of the canonical map $h_M:M\to \head M$ is the intersection $\rad(M)$ of the maximal $\K$-submodules of $M$.  Assume that for some positive integers $m(\l+k\d)$ we have
\begin{equation}
\head(M)=\bigoplus_{\l\in\P^+,~k\in a_0^{-1}\Z}\pi^*\V(\l+k\d)^{\oplus m(\l+k\d)}.
\end{equation}
Denote
\begin{equation}
P_M=\bigoplus_{\l\in\P^+,~k\in a_0^{-1}\Z}P(\l+k\d)^{\oplus m(\l+k\d)}\quad\text{and}\quad W_M=\bigoplus_{\l\in\P^+,~k\in a_0^{-1}\Z}W(\l+k\d)^{\oplus m(\l+k\d)} .
\end{equation}
Since $P_M$ is projective, there is a canonical map $\tilde{h}_M$ making the diagram commutative
\begin{diagram}
                              &                                   & P_M \\
                              &   \ldTo^{\tilde{h}_M}               &  \dTo                   \\
M                           & \rTo_{h_M\ \ }              &\head(M) & \rTo & 0
\end{diagram}
The following Lemma emulates the result in \cite{BBCKL}*{\S 4.4}.
\begin{Lm} \label{L1}
With the notation above, $\tilde{h}_M$ is surjective.
\end{Lm}
\begin{proof}
From the commutativity of the diagram we obtain that $\rad(M)+\tilde{h}_M(P_M)=M$. Since $\rad(M)$ is superfluous we obtain $\tilde{h}_M(P_M)=M$.
\end{proof}
\begin{Def} Let $M$ be an object in ${}_{\K}\M$ and let $\l\in\P^+$. We say that $M$ is $\l$-isotypical if $\wt(M)\subseteq{\wtbar}(\V(\l))$ and $$\head(M)=\bigoplus_{k\in a_0^{-1}\Z}\pi^*\V(\l+k\d)^{\oplus m(\l+k\d)}.$$
\end{Def}
\begin{Rem} Alternatively, the second condition in the definition can be substituted with the condition that $M$ is generated by $\bigoplus_{k\in a_0^{-1}\Z}M_{\l+k\d}$. The direct implication follows from the fact that $$N:=U(\K)\cdot\bigoplus_{k\in a_0^{-1}\Z}M_{\l+k\d}$$ maps surjectively onto $\head(M)$ which implies that $N+\rad(M)=M$ and hence $N=M$ since $\rad(M)$ is superfluous. The converse follows from the fact that $\head(M)$ would have to be generated by the same weight spaces. The local and global Weyl modules indexed  by $\l+k\d$ are examples of $\l$-isotypical objects.
\end{Rem}
The following result is implicit in \cite{BBCKL}*{\S 4.3, \S 4.5}.

\begin{Lm} \label{L2}
If $M$ is $\l$-isotypical then $\tilde{h}_M$ factors through $W_M$ and
\begin{align*}
m(\l+s\d)&=\dim_\C\Hom_\K(M,\pi^*\V(\l+s\d))\\
&=\dim_\C\Hom_\K(M, \Wloc(-w_\circ(\l)-s\d)^*).
\end{align*}
\end{Lm}
\begin{proof}
Let  $I$ be an index set and $I\to\Z, i\mapsto k_i$ a map such that $\head(M)=\bigoplus_{i\in I}\pi^*\V(\l+k_i\d)$. Denote $$M_i:=\tilde{h}_M(P(\l+k_i\d)).$$
Then, the induced map $\tilde{h}_i:P(\l+k_i\d)\to M_i$ is a quotient and $\wt(M_i)\subseteq {\wtbar}(\V(\l))$. From Proposition \ref{P3}(i), the map $\tilde{h}_i$ factors through $W(\l+k_i\d)$. The direct sum of the maps
$$
W(\l+k_i\d)\to M_i\hookrightarrow M
$$
is the descent of $\tilde{h}_M$. From the surjective maps {
$$
W_M  \to  M  \to   \head(M)
$$
}
we obtain injections
$$
\Hom_\K(\head(M),\pi^*\V(\l+s\d))\to \Hom_\K(M,\pi^*\V(\l+s\d))\to \Hom_\K(W_M, \pi^*\V(\l+s\d)).
$$
From Proposition \ref{P2}(i) and Proposition \ref{P3}(iv) we conclude that
$$
\dim_\C\Hom_\K(\head(M),\pi^*\V(\l+s\d))=m(\l+s\d)=\dim_\C\Hom_\K(W_M, \pi^*\V(\l+s\d)),
$$
which implies that
$$m(\l+s\d)=\dim_\C\Hom_\K(M,\pi^*\V(\l+s\d)).$$

{For the second equality, note that $g\in\Hom_\K(M, \Wloc(-w_\circ(\l)-s\d)^*)$ is determined by its restriction to $ \oplus_{i\in I} M_{\l+k_i\d}$ and therefore, since $ \oplus_{k\in a_0^{-1}\Z} (\Wloc(-w_\circ(\l)-s\d)^*)_{\l+k\d}=(\Wloc(-w_\circ(\l)-s\d)^*)_{\l+s\d},$ by its restriction to $M_{k+s\d}$. Consequently, the image of $g$ is contained in the $\K$-submodule generated by $(\Wloc(-w_\circ(\l)-s\d)^*)_{\l+s\d}$ which, by Proposition \ref{P4}iii) is $$\soc(\Wloc(-w_\circ(\l)-s\d)^*)=(\pi^*\V(-w_\circ(\l)-s\d))^*\cong \pi^*\V(\l+s\d).$$ Hence, we have
$$\dim_\C\Hom_\K(M, \Wloc(-w_\circ(\l)-s\d)^*) =\dim_\C\Hom_\K(M, \pi^*\V(\l+s\d)),$$
which is precisely our claim.}
\end{proof}

%%%%%%%%%%%%%%%%%%%%%%%%%%%%%%%%%%%%%%%%%%%%%%%%%%%%%%%%%%
\section{Macdonald polynomials}

\subsection{} For an irreducible affine root system $R$, Macdonald (for $\Rbar$ reduced) and Koornwinder (for $\Rbar$ non-reduced) constructed a family of Weyl group invariant polynomials which depend rationally on a parameter $q$ and a number of $t$ parameters.

The Koornwinder polynomials depend on five $t$ parameters in rank  two or greater and on four parameters in rank one. We will be interested in the connection between Koornwinder polynomials and Demazure module characters for $\Gaff$ of type $A_{2n}^{(2)}$. For this purpose two of the $t$ parameters are specialized (see \cite{IonNon}*{\S 3.2}) leaving us with a family of polynomials depending on a parameter $q$ and as many $t$ parameters as root lengths in $\Rbar$, a situation which is consistent with the dependence of parameters of the Macdonald polynomials.

In establishing the connection between Macdonald and Koornwinder polynomials  with Demazure module characters it is necessary to study the limit of these polynomials as the $t$ parameters approach infinity. Since it is enough for our purposes, and to avoid introducing the complex notation necessary to describe the full-parameter objects, we will only work with Macdonald polynomials for which all the $t$ parameters are set to be equal and with Koornwinder polynomials for which two of the $t$ parameters are specialized according to \cite{IonNon}*{\S 3.2} and the remaining ones are set to be equal.  In what follows we will leave aside the distinction between reduced and non-reduced root systems and the fact that we work in a particular situation and we will call them all Macdonald polynomials.
%%%%%%%%%%%%%%%%%%%%%%%%%%%%%%%%%%%%%%%%%%%%%%%%%%%%%%%%%%
\subsection{} Let $\F_{q,t}$ be the field of rational functions in $q^{a_0^{-1}}$ and $t$ with rational coefficients and $\F_q$ its subfield of rational functions in $q^{a_0^{-1}}$. The algebra $\R_{q,t}=\F_{q,t}[e^\l;\l\in \P]$ are the $\F_{q,t}$-group algebra of the lattice $\P$, while $\R_q=\F_q[e^\l;\l\in \P]$ and $\R=\Rat[e^\l;\l\in \P]$ is  the $\F_q$-group algebra of $\P$ and the $\Rat$-group algebra of $\P$, respectively. The elements $e^\l$, $\l\in\P$ are regarded as characters of the compact torus $\H_\Re/\Q^\vee$ and  multiplication is given by $e^\l \cdot e^\mu = e^{\l+\mu}$. For  $\L\in \P\oplus a_0^{-1}\Z\d$ define  $e^\L\in\R_q$ by setting $e^\d=q^{-1}$.  In particular, the characters \eqref{Demchr} of the affine level one Demazure modules can be regarded as elements of $\R_q$. For any formal expansion $f$ in the characters of the torus $\H_\Re/\Q^\vee$ the coefficient of $e^0$ in the expansion is called the constant term of $f$ and will denoted by $[f]$. The action of $\W$ on $\P$ extends linearly to an action of $\W$ on $\R_{q,t}$; we denote the subalgebra of $\W$-invariants by $\R_{q,t}^{\W}$ and use the corresponding notation for $\R_q$ and $\R$.

%%%%%%%%%%%%%%%%%%%%%%%%%%%%%%%%%%%%%%%%%%%%%%%%%%%%%%%%%%
\subsection{}
The involution of $\F_{q,t}$ which inverts each of the parameters $q,t$ extends to an involution $\overline{~\cdot~}$ on the algebra $\R_{q,t}$ which sends each $e^\l$ to $e^{-\l}$. Let
\begin{equation}
\nabla(q,t)=\prod_{\a\in \Rring^-} \frac{1-e^{\a}}{1-t^{-1}e^{\a}}\prod_{\a\in R^{re,+}} \frac{1-e^{\a}}{1-t^{-1}e^{\a}}
\end{equation}
which should be seen as a formal series in the elements $e^\l$, $\l\in\Q$ with coefficients in $\Z[t^{-1}][[q^{-1}]]$ (power series in $q^{-1}$ with coefficients polynomials in $t^{-1}$). Furthermore, it is both $\W$ and $\overline{~\cdot~}$-invariant.

 In the special case when $t=q^k$ and $k$ is a positive integer this function was introduced by Macdonald in \cite{MacOrt} and used to define a family of orthogonal polynomials associated to root systems and depending on the parameters $q$ and $t$. Note that in this case $\nabla(q,q^k)$ is given by a finite product.

The constant term $[\nabla(q,t)]$ of $\nabla(q,t)$ was the subject of conjectures by  Macdonald, and  Koornwinder, later to be proved by Cherednik \cite{CheDou} (for reduced $\Rbar$) and Gustafson \cite{GusGen} (for nonreduced $\Rbar$) respectively. The Macdonald kernel is defined as
\begin{equation}
\Delta(q,t)=\frac{\nabla(q,t)}{[\nabla(q,t)]} .
\end{equation}
It is $\W$--invariant,  $\overline{~\cdot~}$-invariant, has constant term equal to one, and   can be expressed as a formal series in the characters of  $\H_\Re/\Q^\vee$ with coefficients $\F_{q,t}$ (see, for example, \cite{MacAff}*{(5.1.10)}).

For $f,g\in\R_{q,t}^{\W}$ define
\begin{equation}
\<f,g\>_{q,t}:=\left[f\bar g \Delta(q,t)\right].
\end{equation}
This defines a scalar product on $\R_{q,t}^{\W}$, called the Macdonald scalar product, which is Hermitian with respect to the involution $\overline{~\cdot~}$.
%%%%%%%%%%%%%%%%%%%%%%%%%%%%%%%%%%%%%%%%%%%%%%%%%%%%%%%%%%
\subsection{} An $\F_{q,t}$-basis of $\R_{q,t}^{\W}$ is given by $\{m_\l\}_{\l\in\P^+}$, where
\begin{equation}
m_\l=\sum_{\mu\in \W(\l)} e^\mu
\end{equation}
is the $\W$-invariant monomial corresponding to $\l$.

The symmetric Macdonald polynomials $\{P_\l(q,t)\}_{\l\in \P^-}$ are uniquely defined by the following properties
\begin{enumerate}[label={\alph*)}]
\item $P_\l-m_{w_\circ(\l)}\in \bigoplus_{w_\circ(\l)>\mu\in\P^+} \F_{q,t} m_\mu,$
\item $\<P_\l,P_\mu\>_{q,t}=0$, for $\l\neq \mu.$
\end{enumerate}
In other word, they form a triangular, orthogonal basis of $\R_{q,t}^{\W}$. Their square norms $\| P_\l \|_{q,t}^2$ were computed in \cite{CheDou}*{Theorem 5.1} for $\Rbar$ reduced  and
the conjectured formula in the case of  $\Rbar$ non-reduced follows by combining the results of  \cite{DieSel}*{\S 7.2} and \cite{SahNon}*{Corollary 7.5}, or from \cite{StoKoo}. We refer to \cite{MacAff}*{\S5.8} for a uniform treatment. To avoid introducing more notation we refrain from giving  the full statement; the norm formula in a limiting case is part of Theorem \ref{T4}.

As a consequence of the two defining properties of the Macdonald polynomials we immediately obtain that
\begin{equation}
\overline{P_\l(q,t)}=P_{-w_\circ(\l)}(q,t).
\end{equation}
%%%%%%%%%%%%%%%%%%%%%%%%%%%%%%%%%%%%%%%%%%%%%%%%%%%%%%%%%%
\subsection{} The  following result was proved in \cite{IonNon}*{Theorem 1, Theorem 4.2} for $\Gaff$ of type I. In the case when $\Gaff$ is of type II, part (i) follows from the combinatorial formula in \cite{RamYip}*{Theorem 3.4}.
\begin{Thm}\label{T3}
Let $\Gaff$ be an affine Lie algebra and  let $\l$ be an anti-dominant weight. Then
\begin{enumerate}[label={(\roman*)}]
\item The limit $P_\l(q)=\lim_{t\to\infty}P_\l(q,t)$ exists.
\item If $\Gaff$ is of type I  then
$$
P_\l(q)=\chi(D(\l).
$$
\end{enumerate}
\end{Thm}

Since the polynomials $P_\l(q,t)$ form a basis of $\R_{q,t}^{\W}$ orthogonal with respect to the scalar product $\<\cdot,\cdot\>_{q,t}$ it is natural to ask if such a property holds for the polynomials $P_\l(q)$ with respect to the space $\R^{\W}_q$ and a suitable degeneration of the scalar product $\<\cdot,\cdot\>_{q,t}$ as $t\to\infty$.  The definition of the Macdonald scalar product involves the involution $\overline{~\cdot~}$ on $\R_{q,t}^{\W}$, which inverts the parameter $t$ and hence is  inconsistent with  taking the limit $t\to \infty$. However, we can try to examine the limit as $t\to\infty$ of
\begin{equation}
\<P_\l(q,t),P_\mu(q,t)\>_{q,t}=\left[P_\l(q,t)\overline{P_\mu(q,t)}\Delta(q,t)\right].
\end{equation}
The limit as $t$ approaches infinity of $P_\l(q,t)$,  $\overline{P_\mu(q,t)}$ and $\Delta(q,t)$ exists and equals $P_\l(q)$, $P_{-w_\circ(\mu)}(q)$ and $\Delta(q,\infty)$, respectively. This observation leads us to the following construction.

Let $\iota$ be the involution of $\R_{q,t}$ which fixes the parameters $q$ and $t$ and, for any weight $\l$, sends $e^\l$ to $e^{-w_\circ(\l)}$. We will use the notation $f^\iota:=\iota(f)$ for any element $f$ of $\R_{q,t}$. It is natural to define the following symmetric scalar product on $\R_{q,t}^{\W}$. For $f$ and $g$ in $\R_{q,t}^{\W}$ let
\begin{equation}
\<f,g\>_{q,t}^\prime:=\left[ fg^\iota\Delta(q,t)\right].
\end{equation}

We have proved the following.
\begin{Prop}\label{P5}
The polynomials $\{P_\l(q,t)\}_{\l\in\P^-}$ form a basis of $\R_{q,t}^{\W}$ which is orthogonal with respect to the scalar product $\<\cdot,\cdot\>^\prime_{q,t}$. Their square norms coincide with the square norms with respect to the scalar product $\<\cdot,\cdot\>_{q,t}$.
\end{Prop}
The scalar product $\<\cdot,\cdot\>^\prime_{q,t}$ behaves well with respect to taking the limit $t\to\infty$. This allows us to define the following symmetric scalar product on $\R_q^{\W}$. For $f$ and $g$ in $\R_q^{\W}$ let
\begin{equation}\label{qscalarproduct}
\<f,g\>_{q}:=\left[ fg^\iota \Delta(q,\infty)\right] .
\end{equation}
As an immediate consequence of Theorem \ref{T3}, Proposition \ref{P5}, and the square norm formulas for Macdonald polynomials \cite{CheMac}*{(4.12)} we have the following.
\begin{Thm}\label{T4}
The polynomials $\{P_\l(q)\}_{\l\in\P^-}$ form a basis of $\R_q^{\W}$ which is orthogonal with respect to the scalar product $\<\cdot,\cdot\>_{q}$. Their square norms are given by
$$
\|P_\l(q) \|^2_q=\left(\prod_{i=1}^{n-1}\prod_{j=0}^{-(\l,\a_i^\vee)-1}(1-q^{-(j+1)r_{\a_i}})\right)\left(\prod_{j=0}^{-(\l,\a_n^\vee)-1}(1-q^{-(j+1)\frac{r_{\a_n}}{a_0}})\right)
$$
\end{Thm}
%%%%%%%%%%%%%%%%%%%%%%%%%%%%%%%%%%%%%%%%%%%%%%%%%%%%%%%%%%
\subsection{} Let us make explicit the ingredients of the scalar product $\<\cdot,\cdot\>_{q}$. Taking the appropriate limit in the constant term computation \cites{CheDou, GusGen}
we obtain
\begin{equation}
\begin{aligned}\label{Deltadef}
&\nabla(q,\infty)=\prod_{\a\in \Rring^-} (1-e^{\a})\prod_{\a\in R^{re,+}} (1-e^{\a}),\\
& [\nabla(q,\infty)]=|\W|\left(\prod_{i=1}^{n-1}\prod_{j=1}^{\infty}(1-q^{-jr_{\a_i}})^{-1}\right)\left(\prod_{j=1}^{\infty}(1-q^{-j\frac{r_{\a_n}}{a_0}})^{-1}\right),\\
& \Delta(q,\infty)=\frac{1}{|\W|}\left(\prod_{i=1}^{n-1}\prod_{j=1}^{\infty}(1-q^{-jr_{\a_i}})\right)\left(\prod_{j=1}^{\infty}(1-q^{-j\frac{r_{\a_n}}{a_0}})\right)\prod_{\a\in \Rring^-} (1-e^{\a})\prod_{\a\in R^{re,+}} (1-e^{\a}).
\end{aligned}
\end{equation}
For $f,g\in\R_q^{\W}$ denote
\begin{equation}
\<f,g\>:=\left[fg^\iota \Delta(\infty,\infty)\right].
\end{equation}
 Note that when restricted to $\R^{\W}$ this is the usual scalar product on $\R^{\W}$, the representation ring of $\G$. Let us also consider
\begin{equation}\label{Sdef}
S(q):=\frac{\Delta(q,\infty)}{\Delta(\infty,\infty)}.
\end{equation}
For $\l$ anti-dominant weight, to be consistent with the limit $q\to\infty$ of $P_\l(q)$, we denote by $\chi_\l$ the character of $\V(w_\circ(\l))$ which is the irreducible highest weight representation of $\G$ with highest weight $w_\circ(\l)$. We end this section with the following identity.
\begin{Prop}\label{P6}
Let $\l,\mu$ be anti-dominant weights. Then,
$$
\<\frac{\chi_\l}{S(q)},P_\mu(q)\>_q=\<P_\mu(q),\chi_\l\>.
$$
\end{Prop}
\begin{proof}
Straightforward from the definition of the two scalar products.
\end{proof}
%%%%%%%%%%%%%%%%%%%%%%%%%%%%%%%%%%%%%%%%%%%%%%%%%%%%%%%%%%
\section{BGG reciprocity}
\subsection{} Let $M$ be an object in ${}_\K\M$. Following \cite{BBCKL}*{\S4.2} we construct a $\K$-module filtration of $M$. Let $\mu_0$ be a minimal element (with respect to the dominance partial order) of $\wt(M)\cap\P^+$. For $i\geq 0$ choose $\mu_{i+1}$ a minimal element in the set $S_i=\{\mu\in\wt(M)\cap\P^+~|~\mu\not\leq\mu_j,~0\leq j\leq i\}$ and define
\begin{equation}
M_i:=\sum_{\nu\in S_i\cup\{\mu_i\},~ s\in a_0^{-1}\Z}U(\K)\cdot M_{\nu+s\d} .
\end{equation}
By construction, $\{M_i\}_{i\geq 0}$ is a decreasing filtration of $M$ and, since the $d$-eigenspaces of $M$ are finite dimensional, $\bigcap_{i\geq 0} M_i=0$. Furthermore, for any $i\geq 0$, the quotient $M_i/M_{i+1}$ is $\mu_i$-isotypical. The associated graded module with respect to the filtration $\{M_i\}_{i\geq 0}$ is
\begin{equation}
\gr M:=\bigoplus_{i\geq 0} M_i/M_{i+1} .
\end{equation}
For the following we refer to \cite{BBCKL}*{Lemma 4.3}; we include the proof for the reader's convenience.
\begin{Prop}\label{P7}
With the notation above, there is a canonical surjective morphism
$$
\bigoplus_{i\geq 0}W_{M_i/M_{i+1}}\to \gr M.$$
Moreover,  $W_{M_i/M_{i+1}}=\bigoplus_{s\in a_0^{-1}\Z}W(\mu_i+s\d)^{m(\mu_i+s\d)}$ with  $m(\mu_i+s\d)=\dim_\C\Hom_\K(M,\Wloc(-w_\circ(\mu_i)-s\d)^*)$.
\end{Prop}
\begin{proof}
The result follows from Lemma \ref{L1} and Lemma \ref{L2}. The only fact requiring justification is the formula for $m(\mu_i+s\d)$. 
{For this, remark that, taking into account the possible weights of $\Wloc(-w_\circ(\mu_i)-s\d)^*$, a map $f\in \Hom_\K(M, \Wloc(-w_\circ(\mu_i)-s\d)^*)$ must be identically zero on $M_{i+1}$. Therefore, we need to show that  $m(\mu_i+s\d)=\dim_\C\Hom_\K(M/M_{i+1},\Wloc(-w_\circ(\mu_i)-s\d)^*)$. This would follow if we show that the restriction map 
$$
\Hom_\K(M/M_{i+1},\Wloc(-w_\circ(\mu_i)-s\d)^*) \to \Hom_\K(M_i/M_{i+1},\Wloc(-w_\circ(\mu_i)-s\d)^*).
$$
is bijective.}

{The restriction map is injective. Indeed, assume that $f\in\Hom_\K(M/M_{i+1},\Wloc(-w_\circ(\mu_i)-s\d)^*)$ such that its restriction to $M_i/M_{i+1}$ is identically zero. In particular, the image of $f$ does not contain elements of weight $\mu_i+s\d$ and hence it does not contain $(\pi^*\V(-w_\circ(\mu_i)-s\d))^*$. Since  $(\pi^*\V(-w_\circ(\mu_i)-s\d))^*$ is the unique simple submodule of $\Wloc(-w_\circ(\mu_i)-s\d)^*$, the map $f$ must be identically zero.}

{
The restriction map is also surjective. Indeed, let $g\in \Hom_\K(M_i/M_{i+1},\Wloc(-w_\circ(\mu_i)-s\d)^*)$. Since $P(-w_\circ(\mu_i)-s\d)$ is a projective object in ${}_{\K}\M$, $P(-w_\circ(\mu_i)-s\d)^*$ is an injective object in ${}_{\K}\M$ and therefore the map \begin{diagram}M_i/M_{i+1} &\rTo^{\ \ g\ \ } & \Wloc(-w_\circ(\mu_i)-s\d)^* & \rTo  & P(-w_\circ(\mu_i)-s\d)^*\end{diagram} admits an extension $\tilde{g}: M/M_{i+1}\to P(-w_\circ(\mu_i)-s\d)^*$. Moreover, $\wt(M/M_{i+1})\cap \{\nu~|~\mu_i<\nu\}=\emptyset$, and therefore the image of $\tilde{g}$ is contained in $W(-w_\circ(\mu_i)-s\d)^*$, which by Proposition \ref{P3}ii) is the maximal submodule of $P(-w_\circ(\mu_i)-s\d)^*$ satisfying this condition. Furthermore, $M_i/M_{i+1}$ is the ${\mu_i}$-isotypical component of $M/M_{i+1}$ and since $g(M_i/M_{i+1})$ is contained in $\Wloc(-w_\circ(\mu_i)-s\d)^*$ so is the image of $\tilde{g}$ because  $\Wloc(-w_\circ(\mu_i)-s\d)^*$ is the maximal submodule of $W(-w_\circ(\mu_i)-s\d)^*$ with a one-dimensional $\mu_i$-isotypical component. To conclude, $\tilde{g}\in \Hom_\K(M/M_{i+1},\Wloc(-w_\circ(\mu_i)-s\d)^*)$ and its restriction to $M_i/M_{i+1}$ is $g$.}
\end{proof}
%%%%%%%%%%%%%%%%%%%%%%%%%%%%%%%%%%%%%%%%%%%%%%%%%%%%%%%%%%
\subsection{}

Whether or not the canonical morphism of Proposition \ref{P7} is an isomorphism can be determined by comparing the graded dimensions (with respect to the grading given by the eigenspaces of the scaling element) of the two objects or, even better, by comparing their $\K(0)$-characters. For this purpose we collect below some characters which, as in \S3.2, we regard as elements of $\R_q^{\W}$. From this point of view $q^{-a_0^{-1}}$ is capturing the grading.
\begin{Thm}\label{P8ii}
Let $\Gaff$ be an affine Lie algebra and let $\l+k\d\in P_\K^+$. Then,
$$\chi(\Wloc(\l+k\d))=q^{-k}P_{w_\circ(\l)}(q).$$
\end{Thm}
\begin{proof}
For $\Gaff$ of type I, the claim follows directly from Theorem \ref{T2} and Theorem \ref{T3}(ii). For $\Gaff$ of type II, the claim is proved as follows. {First, \cite{NaoWey}*{Theorem 9.2} expresses $\chi(\Wloc(\l))$ as the generating function of a finite graded crystal, whose objects are projected level-zero affine Lakshimai-Seshadri paths, denoted there by $\mathbb{B}(\l)_{cl}$. Second, \cite{LenFro}*{Theorem 2.6} (see also \cite{LNSSS}*{Proposition 9.8}) expresses $P_{w_\circ(\l)}(q)$ as the generating function of another finite graded crystal called the quantum path model, whose objects are admissible subsets of a lexicographic $\l$-chain \cite{LenLub}, denoted by $\mathcal{A}(\l)$. Third,  \cite{LNSSS}*{Theorem 3.3, Proposition 8.6, Corrolary 9.4} provide an explicit graded crystal isomorphism between $\mathbb{B}(\l)_{cl}$ and $\mathcal{A}(\l)$. Therefore, $P_{w_\circ(\l)}(q)$ and  $\chi(\Wloc(\l))$ coincide.
}
\end{proof}

\begin{Prop}\label{P8} Let $\l+k\d\in P_\K^+$. Then,
\begin{enumerate}[label={(\roman*)}]
\item $\chi(P(\l+k\d))=q^{-k}\chi_{w_\circ(\l)}/S(q)$.
\item $\chi(W(\l+k\d))=q^{-k}P_{w_\circ(\l)}(q)/\|P_{w_\circ(\l)}(q)\|_q^2$.
\end{enumerate}
\end{Prop}
\begin{proof}
By the PBW theorem, as a $\K(0)$-module $P(\l+k\d)$ is isomorphic to $S(\K_+)\otimes \V(\l+k\d)$, where $S(\K_+)$ is the symmetric algebra of $\K_+$. Keeping in mind that we use $q^{-a_0^{-1}}$ to capture the grading, the $\K(0)$-character of $S(\K_+)$ is easily seen to be
\begin{equation}
\left(\prod_{i=1}^{n-1}\prod_{j=1}^{\infty}(1-q^{-jr_{\a_i}})\right)^{-1}\left(\prod_{j=1}^{\infty}(1-q^{-j\frac{r_{\a_n}}{a_0}})\right)^{-1}\left(\prod_{\a\in R_\K\setminus (\Rring\cup R^{+,im})} (1-e^{\a})\right)^{-1} .
\end{equation}
Note that, the first two factors account for $S(\bigoplus_{\a\in R^{im,+}}\Gaff_{\a})$. Necessary for the identification are the multiplicities of the imaginary weights which are described, for example, in \cite{KacInf}*{Corollary 8.3}. The first claim follows now from \eqref{Deltadef} and \eqref{Sdef}.

The last claim follows from Theorem \ref{P8ii}, Corollary \ref{P4iii}, Theorem \ref{T1}(iii), and Theorem \ref{T4}.
\end{proof}

\begin{Thm}\label{T5} Let $\l\in P_\K^+$. Then,  $\bigoplus_{i\geq 0}W_{P(\l+k\d)_i/P(\l+k\d)_{i+1}}$ and $\gr P(\l+k\d)$ are canonically isomorphic.
\end{Thm}
\begin{proof} As remarked above, it is enough to compare the graded characters of the two objects. The $\K(0)$-character of $\gr P(\l+k\d)$ equals the $\K(0)$-character of $P(\l+k\d)$ which is computed in Proposition \ref{P8}. Furthermore, using Theorem \ref{T4}, we can expand the character as
$$
\chi(P(\l+k\d))=\sum_{\mu\in \P^+}q^{-k}\<\frac{\chi_{w_\circ(\l)}}{S(q)},P_{w_\circ(\mu)}(q)\>_qP_{w_\circ(\mu)}(q)/\|P_{w_\circ(\mu)}(q)\|_q^2 .
$$

On the other hand, the character of  $\bigoplus_{i\geq 0}W_{P(\l+k\d)_i/P(\l+k\d)_{i+1}}$  is
$$
\sum_{i\geq 0}\left(\sum_{s\in a_0^{-1}\Z}m(\mu_i+s\d)q^{-s}\right)P_{w_\circ(\mu_i)}(q)/\|P_{w_\circ(\mu_i)}(q)\|_q^2 .
$$
Starting from Lemma \ref{L2}, using the $\ind-\res$ adjunction and Proposition \ref{P6}, we obtain that
\begin{align*}
\sum_{s\in a_0^{-1}\Z}q^{-s}m(\mu_i+s\d)&=\sum_{s\in a_0^{-1}\Z}q^{-s}\dim_\C\Hom_\K(\ind_{\K(0)}^\K\V(\l+k\d),\Wloc(-w_\circ(\mu_i)-s\d)^*)\\
&=\sum_{s\in a_0^{-1}\Z}q^{-s}\dim_\C\Hom_{\K(0)}(\V(\l+k\d),\res^\K_{\K(0)}\Wloc(-w_\circ(\mu_i)-s\d)^*)\\
&=\sum_{s\in a_0^{-1}\Z}q^{-s}\dim_\C\Hom_{\K(0)}(\res^\K_{\K(0)}\Wloc(-w_\circ(\mu_i)-s\d),\V(-w_\circ(\l)-k\d))\\
&=\sum_{s\in a_0^{-1}\Z}q^{-s}\dim_\C\Hom_{\K(0)}(\res^\K_{\K(0)}\Wloc(\mu_i-s\d),\V(\l-k\d))\\
&=q^{-k}\sum_{s\in a_0^{-1}\Z}q^{-s+k}\dim_\C\Hom_{\K(0)}(\res^\K_{\K(0)}\Wloc(\mu_i-s\d+k\d),\V(\l))\\
&=q^{-k}\<P_{w_\circ(\mu_i)}(q),\chi_{w_\circ(\l)}\>\\
&=q^{-k}\<\frac{\chi_{w_\circ(\l)}}{S(q)},P_{w_\circ(\mu_i)}(q)\>_q.
\end{align*}
Therefore,
$$
\chi(\bigoplus_{i\geq 0}W_{P(\l+k\d)_i/P(\l+k\d)_{i+1}})=\sum_{i\geq 0}q^{-k}\<\frac{\chi_{w_\circ(\l)}}{S(q)},P_{w_\circ(\mu_i)}(q)\>_qP_{w_\circ(\mu_i)}(q)/\|P_{w_\circ(\mu_i)}(q)\|_q^2.
$$
Comparing this to the character of $P(\l+k\d)$ and keeping in mind that the canonical map
$$
\bigoplus_{i\geq 0}W_{P(\l+k\d)_i/P(\l+k\d)_{i+1}}\to \gr P(\l+k\d)
$$
is surjective, we conclude that the two characters are equal. Hence, the canonical map is an isomorphism.
\end{proof}

%%%%%%%%%%%%%%%%%%%%%%%%%%%%%%%%%%%%%%%%%%%%%%%%%%%%%%%%%%
\subsection{}
The following concept was introduced in \cite{BCM}.
\begin{Def} Let $M$ be an object in ${}_\K\M$. We say that $M$ has a filtration by global Weyl modules if there exists a $\K$-module descending filtration $\{M_i\}_{i\geq 0}$ of $M$ such that $\cap_{i\geq 0}M_i=0$ and for all $i\geq 0$ we have a $\K$-module isomorphism $$M_{i+1}/M_i\cong\bigoplus_{\mu+s\d\in P_\K^+}W(\mu+s\d)^{m(\mu+s\d)}$$
for some non-negative integers $m(\mu+s\d)$.
\end{Def}
\begin{Rem} If $M$ has a filtration by global Weyl modules then, for any $\mu+s\d\in P_\K^+$, the $\K$-module multiplicity of $W(\mu+s\d)$ in $\gr M$ is finite and independent of the filtration. As customary, we denote this multiplicity by $\left[M: W(\mu+s\d)\right]$.
\end{Rem}
Putting together Theorem \ref{T5} and Proposition \ref{P6} we arrive at our main result.
\begin{Thm}\label{T6}
Let $\l+k\d\in P_\K^+$. The $\K$-module $P(\l+k\d)$ has a filtration by global Weyl modules and
$$
\left[P(\l+k\d): W(\mu+s\d)\right]=\left[\Wloc(\mu+s\d):\pi^*\V(\l+k\d)\right].
$$
\end{Thm}

Theorem \ref{T6} was conjectured in  \cite{BCM}*{Conjecture 2.7} for the current algebras $L[t]$, where $L$ is a simple Lie algebra. Our  result shows that the BGG reciprocity holds for the more general definition of current algebras given in Section \ref{currentdef}. Theorem \ref{T6} was proved in type $A_1$ in \cite{BCM}*{Theorem 1} by different methods. Overall, our proof follows the structure of the argument in \cite{BBCKL}*{Theorem 3.6} where Theorem \ref{T6} was proved for $\Gaff$ of type $A_n^{(1)}$.

%%%%%%%%%%%%%%%%%%%%%%%%%%%%%%%%%%%%%%%%%%%%%%%%%%%%%%%%%%
%%%%%%%%%%%%%%%%%%%%%%%%%%%%%%%%%%%%%%%%%%%%%%%%%%%%%%%%%%%%%%%%%%%%%%%%%%%%%

\end{document}